\def\mystyle{}

\if\mystyle l
\documentclass[landscape]{amsart}
\else
\documentclass{amsart}
\fi

\usepackage{begnac}
\usepackage[arrow,curve]{xy}
\everyxy={0;<1cm,0cm>:<0cm,-1cm>::}

\DeclareMathOperator{\Age}{Age}
\DeclareMathOperator{\Stx}{Stx}
\DeclareMathOperator{\Apx}{Apx}

\newcommand{\clc}[1]{\overline{#1^\uparrow}}

\title{Fraïssé limits of metric structures}

\author{Itaï \textsc{Ben Yaacov}}

\address{Itaï \textsc{Ben Yaacov} \\
  Université Claude Bernard -- Lyon 1 \\
  Institut Camille Jordan, CNRS UMR 5208 \\
  43 boulevard du 11 novembre 1918 \\
  69622 Villeurbanne Cedex \\
  France}

\urladdr{\url{http://math.univ-lyon1.fr/~begnac/}}

\thanks{Research supported by the Institut Universitaire de France and ANR project GruPoLoCo (ANR-11-JS01-008).}

\thanks{The author wishes to thank Julien \textsc{Melleray} for many useful discussions, and Todor \textsc{Tsankov} for pushing him to write this paper.}

\svnInfo $Id: Fraisse.tex 1901 2014-09-02 13:09:14Z begnac $
\thanks{\textit{Revision} {\svnInfoRevision} \textit{of} \today}


\begin{document}

\begin{abstract}
  We develop \emph{Fraïssé theory}, namely the theory of \emph{Fraïssé classes} and \emph{Fraïssé limits}, in the context of metric structures.
  We show that a class of finitely generated structures is Fraïssé if and only if it is the age of a separable approximately homogeneous structure, and conversely, that this structure is necessarily the unique limit of the class, and is universal for it.

  We do this in a somewhat new approach, in which ``finite maps up to errors'' are coded by \emph{approximate isometries}.
\end{abstract}

\maketitle

\section*{Introduction}

The notions of Fraïssé classes and Fraïssé limits were originally introduced by Roland \textsc{Fraïssé} \cite{Fraisse:ExtensionAuxRelations}, as a method to construct countable homogeneous (discrete) structures:
\begin{enumerate}
\item Every Fraïssé class $\cK$ has a Fraïssé limit, which is unique (up to isomorphism).
  The limit is countable and ultra-homogeneous (or, in more model-theoretic terminology, quantifier-free-homogeneous).
\item Conversely, every countable ultra-homogeneous structure is the limit of a Fraïssé class, namely, its \emph{age}.
\end{enumerate}
Moreover, the limit is universal for countable $\cK$-structures, namely for countable structures whose age is contained in $\cK$.

Similar results hold for metric structures as well.
Indeed, some general theory of this form is discussed in the PhD dissertation of Schoretsanitis \cite{Schoretsanitis:PhD}.
Independently, Kubiś and Solecki \cite{Kubis-Solecki:GurariiUniqueness} treated the special case of the class of finite dimensional Banach spaces, essentially showing that their Fraïssé limit is the Gurarij space, which is therefore unique and universal, without ever actually uttering the phrase ``Fraïssé limit'' (and in a fashion which is very specific to Banach spaces).
This multitude of somewhat incompatible approaches, reinforced by considerable nagging from Todor Tsankov convinced the author of the potential usefulness of the present paper.

There is one main novelty in the present treatment, compared with earlier treatments of back-and-forth arguments in the metric setting, in that we replace partial maps with \emph{approximate isometries} (which is just a fancy term for bi-Katětov maps).
These allow us to code in a single, hopefully natural, object, notions such as a partial isometry between metric spaces, or even a ``partial isometry only known up to some error term $\varepsilon > 0$''.
On a technical level, approximate isometries are easier to manipulate than, say, partial isometries, and can be freely composed without loss of information.
More importantly, their use simplifies arguments and dispenses with the need for several limit constructions at several crucial points:
\begin{itemize}
\item In the back-and-forth argument.
  The reader is invited to compare the proof of \autoref{thm:FraisseLimitBackAndForth}, which is hardly distinguishable from the argument for discrete structures, with ``traditional'' arguments for metric structures, involving the construction of partial isomorphisms which only extend each other up to some error, as in the proofs of Facts 1.4 and 1.5 of \cite{BenYaacov-Usvyatsov:dFiniteness}.
\item When checking that a structure is a Fraïssé limit, e.g., when proving that such exists, or when proving that the Gurarij space is the limit of finite-dimensional Banach spaces (\autoref{thm:GurarijExistsUnique}).
  Indeed, approximate isometries allow us to define a Fraïssé limit in a manner which is formally weaker than the ``traditional approach'' definition (namely \autoref{cor:FraisseLimitMap}\ref{item:FraisseLimitMapDefinition}).
  The limit constructions required to pass from the weaker definition to the stronger one are then entirely subsumed in the back-and-forth argument referred to above.
\end{itemize}

Of course, some preliminary work is required in order to develop these tools.
However, once this is done, many arguments in metric model theory, not only those present here, can be simplified significantly, so we consider this is worth the effort.
In addition, approximate isometries are essential for a generalisation of metric Fraïssé theory, to appear in a subsequent paper, in which the limit is only unique up to arbitrarily small error (e.g., a Banach space which is almost isometrically unique).

\section{Approximate isometries}
\label{sec:ApproximateIsometry}

Finite partial isomorphisms between structures play a crucial role in classical Fraïssé theory.
For example, homogeneity and uniqueness of the Fraïssé limit are proved using a back-and-forth argument, in which finite partial maps serve as better and better approximations for a desired global bijection.
In the metric setting, one may expect finite partial isometries to play a similar role, coding partial information regarding a desired global isometry.
However, this analogy fails, essentially on the grounds that whereas finite maps define neighbourhoods of global bijections (in the topology of point-wise convergence), finite isometries do not define neighbourhoods of global isometries.
In order to define an open set of isometries we need to restrict to a finite set \emph{and allow for a small error}: if $g\colon X \dashrightarrow X$ is a finite partial isometry and $\varepsilon > 0$, then $\bigl\{h \in \Iso(X) : hx \in B(gx,\varepsilon) \text{ for all } x \in \dom g \bigr\}$ is open and such sets form a basis for the point-wise convergence topology on $\Iso(X)$.

Another deficiency of partial isometries arises when considering compositions.
Say $f\colon X \dashrightarrow Y$ and $g\colon Y \dashrightarrow Z$ are partial isometries, such that $\img f \cap \dom g = \emptyset$, and say $x \in \dom f$ is such that $fx$ is very close to some $y \in \dom g$.
Then we should like to say that $gfx$ is very close to $gy$, but the composition $gf$ is empty and cannot code this information.

In order to remedy either problem we require a more flexible object than a partial isometry, which can say where an element goes, more or less, without having to say exactly where.
These objects will serve us mostly as approximations of actual isometries, whence their name.
The reader may wish to compare with the treatment of bi-Katětov functions in Uspenskij \cite{Uspenskij:SubgroupsOfMinimalTopologicalGroups}.

\begin{dfn}
  \label{dfn:ApproximateIsometry}
  Let $X$, $Y$ and $Z$ denote metric spaces.
  \begin{enumerate}
  \item We say that a function $\psi\colon X \rightarrow [0,\infty]$ is \emph{Katětov} if for all $x,y \in X$ we have $\psi(x) \leq d(x,y) + \psi(y)$ and $d(x,y) \leq \psi(x) + \psi(y)$.
    Unlike Uspenskij (and Katětov) we allow the value $\infty$, observing that a Katětov function is either finite or constantly $\infty$.
  \item
    We say that $\psi\colon X \times Y \rightarrow [0,\infty]$ is an \emph{approximate isometry} from $X$ to $Y$, and write $\psi\colon X \rightsquigarrow Y$, if it is bi-Katětov, i.e., separately Katětov in each argument.
    The special case $\psi = \infty$ is called the \emph{empty approximate isometry}.
  \item
    Given any $\psi\colon X \times Y \rightarrow [0,\infty]$ and $\varphi\colon Y \times Z \rightarrow [0,\infty]$ we define a \emph{composition} $\varphi\psi\colon X \times Z \rightarrow [0,\infty]$ and a \emph{pseudo-inverse} $\psi^*\colon Y \times X \rightarrow [0,\infty]$ by
    \begin{gather*}
      \varphi\psi(x,z) = \inf_{y \in Y} \, \psi(x,y) + \varphi(y,z), \qquad \psi^*(y,x) = \psi(x,y).
    \end{gather*}
  \end{enumerate}
\end{dfn}

An approximate isometry $\psi\colon X \rightsquigarrow Y$ is meant to provide partial information regarding some isometry, possibly between larger spaces.
We shall understand $\psi$ as saying that $x$ must be sent within $\psi(x,y)$ of $y$, so an isometry $f$ is considered to satisfy the constraints prescribed by $\psi$ if $\psi(x,y) \geq d(fx,y)$ for all $x,y$, i.e., if $f = \psi_f \leq \psi$ in the sense of \autoref{dfn:ApproximateIsometryPartialIsometry} below.
Accordingly, another $\varphi \colon X \rightsquigarrow Y$ imposes stronger constraints if and only if $\psi \geq \varphi$.
The rest of our terminology (coarsening, refinement, etc.) should be understand in the context of this interpretation.

\begin{rmk}
  \label{rmk:ApproximateIsometryAmalgamation}
  Let $\psi\colon X \times Y \rightarrow [0,\infty)$ be given, let $Z = X \amalg Y$, and define $d_Z$ extending $d_X$ and $d_Y$ by $d(x,y) = d(y,x) = \psi(x,y)$.
  Then $\psi$ is bi-Katětov (i.e., an approximate isometry) if and only if $d$ is a pseudo-distance on $Z$.
  The reader is advised that, while this interpretation is close to Katětov's original use for such functions, it is quite distant from our intended use, and may therefore be misleading.
\end{rmk}

\begin{lem}
  \label{lem:ApproximateIsometryOperations}
  \begin{enumerate}
  \item The composition and pseudo-inverse of approximate isometries are again approximate isometries.
  \item Composition is associative, and pseudo-inversion acts as an involution: $\psi^{**} = \psi$, $(\varphi\psi)^* = \psi^* \varphi^*$.
  \end{enumerate}
\end{lem}
\begin{proof}
  Let $\varphi\colon X \rightsquigarrow Y$ and $\psi\colon Y \rightsquigarrow Z$.
  Then for each $x \in X$ and $y \in Y$, the function $z \mapsto \varphi(x,y) + \psi(y,z)$ is Lipschitz with constant $1$, and therefore so is $z \mapsto \psi \varphi(x,z) = \inf_y \, \varphi(x,y) + \psi(y,z)$ for any fixed $x$.
  Similarly, for any $x \in X$, $y,y' \in Y$ and $z,z' \in Z$ we have
  \begin{gather*}
    \varphi(x,y) + \psi(y,z) + \varphi(x,y') + \psi(y',z') \geq d(y,y') + \psi(y,z) + \psi(y',z') \geq \psi(\psi',z) + \psi(y',z) \geq d(z,z'),
  \end{gather*}
  whence $\psi\varphi(x,z) + \psi\psi(x,z') \geq d(z,z')$.
  Therefore $\psi\varphi$ is an approximate isometry, and it is clear that$\varphi^*$ is one as well.
  The second item is even easier and is left to the reader.
\end{proof}

The first examples we give of approximate isometries are simply partial isometries, viewed as instances of the former (modulo some obvious identifications: a partial isometry and its restriction to a dense subset of its domain carry the same information, and indeed induce the same approximate isometry).

\begin{dfn}
  \label{dfn:ApproximateIsometryPartialIsometry}
  Let $X$ and $Y$ denote metric spaces.
  \begin{enumerate}
  \item
    To a partial isometry $f\colon X \dashrightarrow Y$ we associate an approximate isometry $\psi_f(x,y) = \inf_{z \in \dom f} d(x,z) + d(fz,y)$.
    We shall mostly ignore the distinction between $f$ and $\psi_f$, denoting the latter by $f$ as well.
  \item
    Let $i\colon X \subseteq X'$, $j\colon Y \subseteq Y'$ isometric embeddings, and let $\psi\colon X \rightsquigarrow Y$.
    Then $j\psi i^*\colon X' \rightsquigarrow Y'$ is called the \emph{trivial extension} of $\psi$ to $X' \rightsquigarrow Y'$.
    When there is no risk of ambiguity, we shall identify an approximate isometry with its trivial extension to any pair of larger spaces.
  \end{enumerate}
\end{dfn}

\begin{lem}
  \label{lem:ApproximateIsometryPartialIsometry}
  \begin{enumerate}
  \item If $f$ is a partial isometry, then the corresponding $\psi_f$ is an approximate isometry.
  \item The approximate isometry $\infty = \psi_\emptyset$ is destructive for composition, and $\id_X$, identified with $\psi_{\id_X} = d_X$, is neutral.
  \item (Pseudo-)inversion is compatible with the identification of partial isometries with approximate ones.
    Similarly for composition $\psi_g\psi_f = \psi_{gf}$ when $\dom g \supseteq \img f$ or $\dom g \subseteq \img f$, and for the natural notion of trivial extension of a partial map to larger sets.
  \end{enumerate}
\end{lem}
\begin{proof}
  Left to the reader.
\end{proof}

This indeed solves both problems described in the beginning of the section.
If $g\colon X \dashrightarrow X$ is a finite partial isometry and $\varepsilon > 0$ then the approximate isometry $g + \varepsilon$ codes ``$g$ up to error $\varepsilon$'', and $\bigl\{h \in \Iso(X) : hx \in B(gx,\varepsilon) \text{ for all } x \in \dom g \bigr\}$ is just $\Iso(X) \cap \Apx^{<g+\varepsilon}(X)$ in the sense of \autoref{dfn:ApproximateIsometrySpace} below.
Similarly, in the situation of composition of partial isometries, if $x \in \dom f$ and $y \in \dom g$ then $\psi_g \psi_f$ prescribes that $x$ be sent no more than $(\psi_g \psi_f)(x,gy) = d(fx,y)$ from $gy$, which is exactly the information we wanted to keep.

\begin{dfn}
  \label{dfn:ApproximateIsometrySpace}
  Let $X$, $Y$ and $Z$ denote metric spaces.
  \begin{enumerate}
  \item
    The space of all approximate isometries from $X$ to $Y$ will be denoted $\Apx(X,Y)$, and equipped with the topology induced from $[0,\infty]^{X \times Y}$.
    When $X = Y$ we let $\Apx(X) = \Apx(X,X)$.
  \item
    For $\psi,\varphi \in \Apx(X,Y)$ we say that $\varphi \leq \psi$ is the comparison holds point-wise, i.e., $\varphi(x,y) \leq \psi(x,y)$ for all $(x,y) \in X \times Y$.
    We then also say that $\psi$ \emph{coarsens} $\varphi$, or that $\varphi$ \emph{refines} $\psi$.
    We define $\Apx^{\leq \psi}(X,Y) = \bigl\{ \varphi \in \Apx(X,Y) : \varphi \leq \psi \bigr\}$.
  \item
    We define $\Apx^{<\psi}(X,Y)$ as the interior of $\Apx^{\leq \psi}(X,Y)$ in $\Apx(X,Y)$.
    If $\varphi \in \Apx^{<\psi}(X,Y)$ we write $\varphi < \psi$ and say that $\psi$ \emph{strictly coarsens} $\varphi$, or that $\varphi$ \emph{strictly refines} $\psi$.
  \item
    For $\cA \subseteq \Apx(X,Y)$ we define its \emph{closure under coarsening} $\cA^\uparrow = \{ \psi \in \Apx(X,Y) : \exists \varphi \in \cA, \, \psi \geq \varphi\}$.
    We observe that $\clc{\cA} = (\clc{\cA})^\uparrow$.
  \end{enumerate}
\end{dfn}

\begin{lem}
  \label{lem:ApproximateIsometryTopology}
  \begin{enumerate}
  \item The space $\Apx(X,Y)$ is compact, and the interpretation of actual isometries as approximate isometries yields a topological embedding $\Iso(X) \subseteq \Apx(X)$.
  \item If $\varphi_\alpha \in \Apx(X,Y)$ is a net then $\limsup \varphi_\alpha$, calculated coordinate-wise in $[0,\infty]^{X \times Y}$, belongs to $\Apx(X,Y)$ as well.
  \item
    \label{item:ApproximateIsometryTopologyComposition}
    Composition is upper semi-continuous, in the sense that the set $\bigl\{ (\psi,\varphi) : \varphi\psi \in \cA \bigr\} \subseteq \Apx(X,Y) \times \Apx(Y,Z)$ is closed whenever $\cA = \clc{\cA} \subseteq \Apx(X,Z)$.
    Equivalently, $(\limsup \psi_\alpha)(\limsup \varphi_\alpha) \geq \limsup (\psi_\alpha \varphi_\alpha)$.
  \end{enumerate}
\end{lem}
\begin{proof}
  The space $\Apx(X,Y)$ is closed in $[0,\infty]^{X \times Y}$ and therefore compact.
  A sub-basic open set $U_{x,y,\varepsilon} = \{g : d(gx,y) < \varepsilon\} \subseteq \Iso(X)$ agrees with $\{\varphi : \varphi(x,y) < \varepsilon\} \cap \Iso(X)$ under $\Iso(X) \subseteq \Apx(X)$.
  Conversely, if $V = \{\varphi : r < \varphi(x,y) < s\} \subseteq \Apx(X)$ and $f \in \Iso(X) \cap V$ then we may assume that $r + \varepsilon < d(fx,y) < s-\varepsilon$ in which case $f \in U_{x,fx,\varepsilon} \subseteq V$.
  This proves the first item.
  That $\limsup \varphi_\alpha$ is also an approximate isometry, and that $(\limsup \psi_\alpha)(\limsup \varphi_\alpha) \geq \limsup (\psi_\alpha \varphi_\alpha)$, follow easily from the definitions.
  The latter, together with $\cA = \clc{\cA}$, implies that $\bigl\{ (\psi,\varphi) : \varphi\psi \in \cA \bigr\}$ is closed.
\end{proof}

\begin{lem}
  \label{lem:ApproximateIsometryStrictCoarsening}
  Let $X$, $Y$ and $Z$ be metric spaces.
  \begin{enumerate}
  \item
    Let $\psi \in U \subseteq \Apx(X,Y)$, with $U$ a neighbourhood of $\psi$.
    Then there exists $\varphi \in U$ such that $\psi < \varphi$.
    In particular, if $\psi < \varphi$ in $\Apx(X,Y)$ and $V \ni \psi$ is open then there exists $\rho \in \Apx(X,Y) \cap V$ such that $\psi < \rho < \varphi$.
  \item
    \label{item:ApproximateIsometryStrictCoarseningCriterion}
    Let $\varphi,\psi \in \Apx(X,Y)$.
    Then $\varphi > \psi$ if and only if there are finite $X_0 \subseteq X$, $Y_0 \subseteq Y$ and $\varepsilon > 0$ such that $\varphi \geq \psi\rest_{X_0 \times Y_0} + \varepsilon$.
    Moreover, in this case there exists $\rho \in \Apx(X_0,Y_0)$ which only takes rational values (on $X_0 \times Y_0$) such that $\psi < \rho < \varphi$.
  \end{enumerate}
\end{lem}
\begin{proof}
  For the first item, we may assume that there are finite sets $X_0 \subseteq X$, $Y_0 \subseteq Y$ and some $\varepsilon > 0$ such that $\varphi \in U$ if and only if $|\varphi(x,y) - \psi(x,y)| < 2\varepsilon$ on $X_0 \times Y_0$.
  Let $\psi_0 = \psi\rest_{X_0 \times Y_0} \in \Apx(X_0,Y_0)$, and let $\varphi = \psi_0 + \varepsilon \in \Apx(X_0,Y_0) \subseteq \Apx(X,Y)$.
  Let
  \begin{gather*}
    V = \{\varphi\colon \varphi(x,y) < \psi(x,y) + \varepsilon \text{ on } X_0 \times Y_0\}.
  \end{gather*}
  Then $\psi \in V \subseteq \Apx^{\leq \varphi}(X,Y)$, so $\psi < \varphi$.

  The rest is easy.
\end{proof}

\section{Metric Fraïssé limits via approximate maps}
\label{sec:Fraisse}

Let us start by fixing a few basic definitions.

\begin{dfn}
  Let $\cL$ be denote a collection of symbols, each being either a \emph{predicate symbol} or a \emph{function symbol} and each having an associated natural number called its \emph{arity}.
  An \emph{$\cL$-structure} $\fA$ consists of a complete metric space $A$, together with,
  \begin{itemize}
  \item For each $n$-ary predicate symbol $R$, a continuous interpretation $R^\fA\colon A^n \rightarrow \bR$.
    It will be convenient to consider the distance as a (distinguished) binary predicate symbol.
  \item For each $n$-ary function symbol $f$, a continuous interpretation $f^\fA\colon A^n \rightarrow A$.
    A zero-ary function is also called a \emph{constant}.
  \end{itemize}

  If $\fA$ is a structure and $A_0 \subseteq A$, then the smallest substructure of $\fA$ containing $A_0$ is denoted $\langle A_0 \rangle$, the substructure \emph{generated} by $A_0$.
  Its underlying set is just the metric closure of $A_0$ under the interpretations of function symbols.

  An \emph{embedding} of $\cL$-structures $\varphi\colon \fA \rightarrow \fB$ is a map which commutes with the interpretation of the language: $R^\fB(\varphi \bar a) = R^\fA(\bar a)$ and $f^\fB(\varphi \bar a) = \varphi f^\fA(\bar a)$ (in particular, $d^\fB(\varphi a,\varphi b) = d^\fA(a,b)$, so an embedding is always isometric).
  A \emph{partial isomorphism} $\varphi\colon \fA \dashrightarrow \fB$ is a map $\varphi\colon A_0 \rightarrow B$ where $A_0 \subseteq A$ and $\varphi$ extends (necessarily uniquely) to en embedding $\langle A_0 \rangle \rightarrow \fB$.
\end{dfn}

\begin{rmk}
  \label{rmk:Continuity}
  The definition given here is more relaxed than definitions given in more general treatments of continuous logic, such as \cite{BenYaacov-Usvyatsov:CFO,BenYaacov-Berenstein-Henson-Usvyatsov:NewtonMS} for the bounded case and \cite{BenYaacov:Unbounded} for the general (unbounded) case, in that we only require plain continuity (rather than uniform), and no kind of boundedness.
  Indeed, let us consider the following properties of a map $f\colon X \rightarrow Y$ between metric spaces, which imply one another from top to bottom:
  \begin{enumerate}
  \item The map $f$ is uniformly continuous.
  \item The map $f$ sends Cauchy sequences to Cauchy sequences (equivalently, $f$ admits a continuous extension to the completions, $\hat f \colon \widehat X \rightarrow \widehat Y$).
    Let us call this \emph{Cauchy continuity}.
  \item The map $f$ is continuous.
  \end{enumerate}
  If $X$ is complete then the last two properties coincide, if $X$ is totally bounded then the first two coincide, and if $X$ is compact then all three do.
  Thus Cauchy continuity is intimately connected with completeness.
  Similarly, uniform continuity is intimately related with compactness: on the one hand, compactness implies uniform continuity (assuming plain continuity), while on the other hand, uniform continuity of the language is a crucial ingredient in the proof of compactness for first order continuous logic (similarly, in unbounded logic, compactness below every bound corresponds to uniform continuity on bounded sets).

  In light of this, and since compactness will not intervene in any way in our treatment, plain continuity on complete spaces will suffice.
  In situations involving incomplete spaces we shall require Cauchy continuity.
\end{rmk}

\begin{dfn}
  \label{dfn:ApproximatelyUltraHomogeneous}
  We say that a separable structure $\fM$ is \emph{approximately ultra-homogeneous} if every finite partial isomorphism $\varphi\colon \fM \dashrightarrow \fM$ is arbitrarily close to the restriction of an automorphism of $\fM$: for every $\varepsilon > 0$ there exists $f \in \Aut(\fM)$ such that $d(\varphi a,fa) < \varepsilon$ for all $a \in \dom \varphi$.
  Equivalently, if $\clc{\Aut(\fM)} \subseteq \Apx(M)$ contains every (finite) partial isomorphism $\varphi\colon \fM \dashrightarrow \fM$.
\end{dfn}

\begin{dfn}
  \label{dfn:Age}
  The \emph{age} of an $\cL$-structure $\fA$, denoted $\Age(\fA)$, is the class of finitely generated structures which embed in $\fA$.
\end{dfn}

Metric Fraïssé theory deals with (ages of) approximately ultra-homogeneous separable structures.
One could, of course, say that a structure $\fM$ is (precisely, rather than approximately) ultra-homogeneous if every isomorphism of finitely generated substructures extends to an isomorphism, but this would make us lose important examples (e.g., the Gurarij space), and in any case it does not seem that a Fraïssé theory can be developed for this stronger notion.
It follows that, whereas classical Fraïssé theory deals with finite partial isomorphism (and their extensions to automorphisms), metric Fraïssé theory must deal with finite partial isomorphisms ``up to some error'', which is by no means a new phenomenon in metric model theory.

The standard approach so far in similar situations, say when carrying out back-and-forth arguments (see for example \cite[Facts 1.4 and 1.5]{BenYaacov-Usvyatsov:dFiniteness}), involves constructing a sequence of finite partial isomorphisms $f_n$ such that each $f_{n+1}$ only extends $f_n$ up to some allowable error $\varepsilon_n$, keeping $\sum \varepsilon_n$ small.
This involves a considerable amount of bookkeeping, limit constructions and other complications.
Replacing ``partial isometries up to error'' with approximate isometries, as suggested in \autoref{sec:ApproximateIsometry}, we manage to avoid these complications, and the metric Fraïssé theory follows quite effortlessly, in almost perfect analogy with its discrete counterpart.

\begin{dfn}
  \label{dfn:HP}
  Let $\cK$ be a class of finitely generated structures.
  \begin{enumerate}
  \item
    By a \emph{$\cK$-structure} we mean an $\cL$-structure $\fA$ such that $\Age(\fA) \subseteq \cK$.
  \item
    We say that $\cK$ has the \emph{HP (Hereditary Property)} if every member of $\cK$ is a $\cK$-structure.
  \item
    Assume that $\cK$ has HP.
    We say that $\cK$ has the \emph{NAP (Near Amalgamation Property)} if for every $\fA,\fB \in \cK$, finite partial isomorphism $f\colon \fA \dashrightarrow \fB$ and $\varepsilon > 0$ there are $\fC \in \cK$ and embeddings $g\colon \fA \rightarrow \fC$, $h\colon \fB \rightarrow \fC$ such that $d(ga,hfa) \leq \varepsilon$ for all $a \in \dom f$, or equivalently, such that (as approximate isometries) $f + \varepsilon \geq h^*g$.
  \end{enumerate}
\end{dfn}

Notice that an age always has HP, and if $\fM$ is approximately ultra-homogeneous then $\Age(\fM)$ has NAP as well.

\begin{dfn}
  \label{dfn:ApproximateIsomorphism}
  Let $\cK$ be a class of finitely generated structures with HP, and let $\fA$ and $\fB$ be $\cK$-structures.
  We define $\Apx_1(\fA,\fB)$ to be the set of all finite partial isomorphisms $f\colon \fA \dashrightarrow \fB$, and
  \begin{gather*}
    \Apx_{2,\cK}(\fA,\fB) = \{ fg : g \in \Apx_1(\fA,\fC) \text{ and } f \in \Apx_1(\fC,\fB) \text{ for some } \fC \in \cK\},
  \end{gather*}
  where composition is in the sense of approximate isometries.
  Notice that if we allowed $\fC$ to be an arbitrary $\cK$-structure we would obtain the same definition, since we can always replace $\fC$ with $\langle\img g \cup \dom f \rangle$.
  Finally, following \autoref{dfn:ApproximateIsometrySpace}, define
  \begin{gather*}
    \Apx_\cK(\fA,\fB) = \overline{\Apx_{2,\cK}(\fA,\fB)^\uparrow}.
  \end{gather*}
  Members of $\Apx_\cK(\fA,\fB)$ are called \emph{($\cK$-intrinsic) approximate isomorphisms}.
  When $\cK$ is clear from the context we usually drop it.

  For $\psi \in \Apx(\fA,\fB)$, we define $\Apx^{<\psi}(\fA,\fB) = \Apx(\fA,\fB) \cap \Apx^{<\psi}(A,B)$.
  We say that $\psi$ is a \emph{strictly approximate isomorphism} if $\Apx^{<\psi}(\fA,\fB) \neq \emptyset$, and let $\Stx(\fA,\fB)$ denote the collection of such $\psi$.
\end{dfn}

Intuitively, approximate isomorphisms are to partial isomorphisms (between members of $\cK$) as approximate isometries are to partial isometries, so in particular every member of $\Apx_1(\fA,\fB)$ should then be considered an approximate isomorphism.
The reason for taking the two-iterate is that $\Apx_1$ may ``miss'' some information: for example, it may happen that $\fA,\fB \in \cK$ are ``close'', as witnessed by some embeddings $\fA \rightarrow \fC$ and $\fB \rightarrow \fC$ with close images (which is captured by $\Apx_2$), even though they have no non-trivial common substructure (so $\Apx_1$ sees nothing).
We also require $\Apx(\fA,\fB)$ to be compact and closed under coarsening (as is $\Apx(A,B)$), whence the definition.
\emph{Strictly} approximate isomorphisms are analogous to \emph{finite} partial isomorphisms in the classical setting, in that they do not fix too much information, leaving an open set of possibilities (clearly, $\Apx^{<\psi}(\fA,\fB)$ contains the relative interior of $\Apx^{\leq \psi}(\fA,\fB)$ in $\Apx(\fA,\fB)$, and one can check that in fact, the two agree).

\begin{lem}
  \label{lem:ApproximateIsomorphism}
  Let $\cK$ be a class of finitely generated structures with HP.
  Let $\fA$ and $\fB$ be $\cK$-structures.
  Then
  \begin{enumerate}
  \item $\Apx(\fA,\fB) = \overline{\Stx(\fA,\fB)}$.
  \item Every partial isomorphism between $\fA$ and $\fB$ belongs to $\Apx(\fA,\fB)$ (see \autoref{rmk:FraisseClass} below for a converse of this).
  \item
    \label{item:ApproximateIsomorphismStrictOver2}
    If $\psi \in \Stx(\fA,\fB)$ then $\Apx_2^{<\psi}(\fA,\fB) \neq \emptyset$.
  \end{enumerate}
\end{lem}
\begin{proof}
  First, let $\varphi \in \Apx(\fA,\fB)$.
  For finite $A_0 \subseteq A$ and $B_0 \subseteq B$ and for $\varepsilon > 0$ we have $\varphi\rest_{A_0 \times B_0} + \varepsilon > \varphi$ by \autoref{lem:ApproximateIsometryStrictCoarsening}\ref{item:ApproximateIsometryStrictCoarseningCriterion}, so $\varphi\rest_{A_0 \times B_0} + \varepsilon \in \Stx(\fA,\fB)$.
  It follows that $\varphi \in \overline{\Stx(\fA,\fB)}$, and the converse inclusion is clear.
  Similarly, every finite partial isomorphism is an approximate isomorphism.
  Taking limits, every partial isomorphism is an approximate isomorphism.
  The last item follows from the definitions.
\end{proof}

\begin{lem}
  \label{lem:NAP}
  Let $\cK$ be a class of finitely generated structures with HP and NAP.
  Then
  \begin{enumerate}
  \item For any $\fA,\fB \in \cK$, $\varphi \in \Apx_2(\fA,\fB)$ and $\varepsilon > 0$ there exist $\fC \in \cK$ and embeddings $f\colon \fA \rightarrow \fC$, $g\colon \fB \rightarrow \fC$ such that $g^* f < \varphi + \varepsilon$.
  \item \label{item:NAPStxAmalgamation}
    For any $\fA,\fB \in \cK$ and $\varphi \in \Stx(\fA,\fB)$ there exist $\fC \in \cK$ and embeddings $f\colon \fA \rightarrow \fC$, $g\colon \fB \rightarrow \fC$ such that $g^* f < \varphi$.
  \item \label{item:NAPStxToEmbedding}
    Let $\fA \in \cK$, let $\fB$ be a $\cK$-structure, and let $\varphi \in \Stx(\fA,\fB)$.
    Then there exists an extension $\fA \subseteq \fC \in \cK$ and a finite partial isomorphism $f\colon \fC \dashrightarrow \fB$ such that $f < \varphi$.
  \item The composition of any two (strictly) approximate isomorphisms between $\cK$-structures is one as well.
  \end{enumerate}
\end{lem}
\begin{proof}
  For the first item, there exists $\fC_0 \in \cK$ and finite partial isomorphisms $f_0\colon \fA \dashrightarrow \fC_0$, $g_0 \colon \fC_0 \dashrightarrow \fB$ such that $\varphi = g_0 f_0$.
  By NAP there are $\fC_1,\fC_2 \in \cK$ and embeddings as in the diagram below such that $f_2^* f_1 \leq f_0 + \varepsilon$ and $g_2^* g_1 \leq g_0 + \varepsilon$.
  Let $X = \img f_0 \subseteq C_0$, a finite set, and we let $h_0\colon \fC_1 \dashrightarrow \fC_2$ be the finite partial isomorphism sending $f_2 X \mapsto g_1 X$, i.e., $h_0 = (g_1 f_2^*)\rest_{f_2 X \times g_1 X} = g_1 \id_X f_2^*$.
  Applying NAP once more we complete the diagram with $h_2^* h_1 \leq h_0 + \varepsilon$.
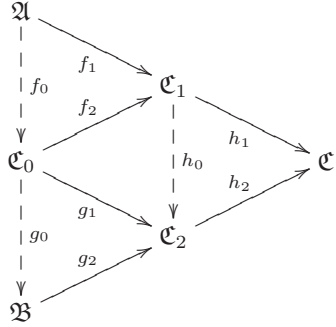
  \begin{gather*}
    \begin{xy}
      (0,0)*+{\fA}="A",
      (0,2)*+{\fC_0}="C0",
      (0,4)*+{\fB}="B",
      (2,1)*+{\fC_1}="C1",
      (2,3)*+{\fC_2}="C2",
      (4,2)*+{\fC}="C",
      \ar^{f_0} @{-->} "A";"C0"
      \ar^{g_0} @{-->} "C0";"B"
      \ar_{f_1} "A";"C1"
      \ar^{f_2} "C0";"C1"
      \ar_{g_1} "C0";"C2"
      \ar^{g_2} "B";"C2"
      \ar^{h_0} @{-->} "C1";"C2"
      \ar_{h_1} "C1";"C"
      \ar^{h_2} "C2";"C"
    \end{xy}
  \end{gather*}
  Now,
  \begin{gather*}
    \varphi + 3\varepsilon
    = g_0 \id_X f_0 + 3\varepsilon
    \geq g_2^* g_1 \id_X f_2^* f_1 + \varepsilon
    = g_2^* h_0 f_1 + \varepsilon
    \geq g_2^* h_2^* h_1 f_1.
  \end{gather*}
  Letting $f = h_1 f_1$ and $g = h_2 g_2$ we obtain $\varphi + 4\varepsilon > g^* f$ (by \autoref{lem:ApproximateIsometryStrictCoarsening}\ref{item:ApproximateIsometryStrictCoarseningCriterion}, with $X_0 \times Y_0 = \dom f_0 \times \img g_0$), which is enough.

  For the second item, by \autoref{lem:ApproximateIsomorphism}\autoref{item:ApproximateIsomorphismStrictOver2} there exists $\psi \in \Apx_2^{<\varphi}(\fA,\fB)$ and we may assume that $\psi + \varepsilon < \varphi$ for some $\varepsilon > 0$.
  We then apply the first item.

  For the third item, by definition of $\Apx_2$ there exists a finitely generated $\fB \supseteq \fB_0 \in \cK$ such that $\varphi \in \Apx_2(\fA,\fB_0)$.
  Applying the second item, there exists $\fC \in \cK$ and embeddings $f\colon \fA \rightarrow \fC$ and $g\colon \fB \rightarrow \fC$ such that $g^* f < \varphi$.
  We may assume that $f$ is an inclusion, and by \autoref{lem:ApproximateIsometryStrictCoarsening}\autoref{item:ApproximateIsometryStrictCoarseningCriterion} there exists a finite set $C_0 \subseteq C$ such that $g^{-1}\rest_{C_0} \in \Apx^{<\varphi}(\fC,\fB)$.
  Finally, by definition of the trivial extension, we may assume that $C_0 \subseteq \img g$ so $g^{-1}\rest_{C_0}$ is a finite partial isomorphism.

  For the last item, let $\varphi_i \in \Stx(\fA_i,\fA_{i+1})$ for $i = 0,1$, where $\fA_i$ are $\cK$-structures.
  By definition there exist $\psi_i \in \Apx_2^{<\varphi_i}(\fA_i,\fA_{i+1})$, say $\psi_i = g_i* f_i$ where $f_i\colon \fA_i \dashrightarrow \fB_i$ and $g_i\colon \fA_{i+1} \dashrightarrow \fB_i$ are finite partial isomorphisms, and $\fB_i \in \cK$.
  By the first item (and the fact that only a finitely generated sub-structure of $\fA_1$ is actually involved) there are $\fC \in \cK$ and embeddings $h_i\colon \fB_i \rightarrow \fC$ such that $h_1^* h_0 < f_1 g_0^* + \varepsilon$:
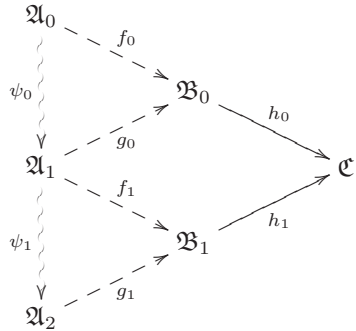
  \begin{gather*}
    \begin{xy}
      (0,0)*+{\fA_0}="A0",
      (0,2)*+{\fA_1}="A1",
      (0,4)*+{\fA_2}="A2",
      (2,1)*+{\fB_0}="B0",
      (2,3)*+{\fB_1}="B1",
      (4,2)*+{\fC}="C",
      \ar_{\psi_0} @{~>} "A0";"A1"
      \ar_{\psi_1} @{~>} "A1";"A2"
      \ar^{f_0} @{-->} "A0";"B0"
      \ar_{g_0} @{-->} "A1";"B0"
      \ar^{f_1} @{-->} "A1";"B1"
      \ar_{g_1} @{-->} "A2";"B1"
      \ar^{h_0} "B0";"C"
      \ar_{h_1} "B1";"C"
    \end{xy}
  \end{gather*}
  It follows that $(h_1 g_1)^* (h_0 f_0) \leq \psi_1 \psi_0 + \varepsilon$, and $\rho = (h_1 g_1)^* (h_0 f_0)$ belongs to $\Apx_2(\fA_0,\fA_2)$.
  Now, since $\varepsilon$ was arbitrary, we could have chosen it so that $\psi_i + \varepsilon < \varphi_i$.
  Again by \autoref{lem:ApproximateIsometryStrictCoarsening}\ref{item:ApproximateIsometryStrictCoarseningCriterion} it follows that $\psi_1 \psi_0 + \varepsilon < \psi_1 \psi_0 + 2\varepsilon \leq \varphi_1 \varphi_0$, so $\rho < \varphi_1 \varphi_0$ and $\varphi_1 \varphi_0 \in \Stx(\fA_0,\fA_1)$.

  When $\varphi_i \in \Apx(\fA_i,\fA_{i+1})$, we have $\varphi_i = \lim_\alpha \varphi_{i,\alpha}$ for nets $\varphi_{i,\alpha} \in \Stx(\fA_i,\fA_{i+1})$, by \autoref{lem:ApproximateIsomorphism}.
  Possibly passing to a sub-net we may assume $\psi = \lim_\alpha \varphi_{1,\alpha} \varphi_{0,\alpha}$ exists (by compactness).
  Then $\psi \in \Apx(\fA_0,\fA_2)$ and $\psi \leq \varphi_1 \varphi_0$ (by \autoref{lem:ApproximateIsometryTopology}\ref{item:ApproximateIsometryTopologyComposition}), which is enough.
\end{proof}

\begin{conv}
  We equip products of metric spaces with the supremum distance, so for two $n$-tuples $\bar a$ and $\bar b$ we have $d(\bar a,\bar b) = \max_i d(a_i,b_i)$.
\end{conv}

\begin{dfn}
  \label{dfn:Kn}
  Let $\cK$ be a class of finitely generated $\cL$-structures.
  For $n \geq 0$, we let $\cK_n$ denote the class of all pairs $(\bar a,\fA)$, where $\fA \in \cK$ and $\bar a \in A^n$ generates $\fA$.
  By an abuse of notation, we shall refer to $(\bar a,\fA) \in \cK_n$ by $\bar a$ alone, and denote the generated structure $\fA$ by $\langle \bar a \rangle$.

  By $\Apx(\bar a,\fB)$ we shall mean those members $\Apx(\langle \bar a \rangle, \fB)$ which extend trivially from an approximate isometry $\bar a \rightsquigarrow B$, and similarly for $\Stx(\bar a,\fB)$, $\Apx(\bar a,\bar b)$, and so on.
  Under HP and NAP, these still compose correctly as per \autoref{lem:NAP}.
\end{dfn}

\begin{dfn}
  \label{dfn:IntrinsicDistance}
  Let $\cK$ be a class of finitely generated structures with NAP.
  We equip $\cK_n$ with a pseudo-distance $d^\cK$ defined by
  \begin{gather*}
    d^\cK(\bar a,\bar b) = \inf_{\psi \in \Stx(\bar a,\bar b)} d(\psi) = \inf_{\psi \in \Apx(\bar a,\bar b)} d(\psi), \qquad \text{where } d(\psi) = \max_i \psi(a_i,b_i).
  \end{gather*}
\end{dfn}

Equivalently, $d(\bar a,\bar b)$ is the infimum of all possible $d(\bar a,\bar b)$ under embeddings of $\langle \bar a \rangle$ and $\langle \bar b \rangle$ into some $\fC \in \cK$.
The triangle inequality is a consequence of \autoref{lem:NAP}.

\begin{dfn}
  \label{dfn:FraisseClass}
  A \emph{Fraïssé class} (of $\cL$-structures) is a class $\cK$ of finitely generated $\cL$-structures having the following properties:
  \begin{itemize}
  \item \emph{HP}.
  \item \emph{JEP (Joint Embedding Property)}: Every two members of $\cK$ embed in a third one.
  \item \emph{NAP}.
  \item \emph{PP (Polish Property)}: The pseudo-metric $d^\cK$ is separable and complete on $\cK_n$ for each $n$.
  \item \emph{CP (Continuity Property)}: Every symbol is continuous on $\cK$.
    For an $n$-ary predicate symbol $P$, this means that the map $K_n \rightarrow \bR$, $\bar a \mapsto P^{\langle \bar a\rangle}(\bar a)$, is continuous.
    For an $n$-ary function symbol $P$, this means that for each $m$, the map $K_{n+m} \rightarrow \cK_{n+m+1}$, $(\bar a, \bar b) \mapsto \bigl( \bar a, \bar b, f^{\langle \bar a, \bar b \rangle}(\bar a) \bigr)$, is continuous.
  \end{itemize}
  We say that $\cK$ is an \emph{incomplete Fraïssé class} if instead of PP \& CP we have:
  \begin{itemize}
  \item \emph{WPP (Weak Polish Property)}: The pseudo-metric $d^\cK$ is separable on $\cK_n$ for each $n$.
  \item \emph{CCP (Cauchy Continuity Property)}: Every symbol is Cauchy continuous on $\cK$ (as per \autoref{rmk:Continuity}).
  \end{itemize}
\end{dfn}

\begin{rmk}
  \label{rmk:FraisseClass}
  We observe that:
  \begin{enumerate}
  \item CP implies that the kernel of $d^\cK$ on $\cK_n$ is exactly the isomorphism relation: $d^\cK(\bar a,\bar b) = 0$ if and only if exists a (necessarily unique) isomorphism $\varphi\colon \langle \bar a \rangle \rightarrow \langle \bar b \rangle$ sending $\bar a \mapsto \bar b$.
    It follows that a partial isometry between $\cK$-structures is an approximate isomorphism if and if it is a partial isomorphism.
  \item Together with PP this implies that a $\cK$-structure generated by a set of cardinal $\kappa$ has density character at most $\kappa + \aleph_0$ (even if the language contains more than $\kappa$ symbols).
    In particular, every member of $\cK$ is separable.
  \item Every Fraïssé class is in particular an incomplete Fraïssé class, and conversely, every incomplete Fraïssé class $\cK$ admits a unique \emph{completion} $\widehat \cK$, consisting of all limits of Cauchy sequences in $\cK$ (that is, in $\cK_n$, as $n$ varies), which is a Fraïssé class.
  \item JEP is equivalent to saying that the empty approximate isometry is always a (strictly) approximate isomorphism.
    Modulo NAP, JEP is further equivalent to there being a unique $\emptyset$-generated (empty, if there are no constant symbols) structure in $\cK$.
  \end{enumerate}
\end{rmk}

\begin{dfn}
  \label{dfn:ApproximateIsometryTotal}
  We say that an approximate isometry $\psi\colon X \rightsquigarrow Y$ is \emph{$r$-total} for some $r > 0$ if $\psi^*\psi \leq \id_X + 2r$, or equivalently, if for all $x \in X$ and $s > r$ there is $y \in Y$ such that $\psi(x,y) < s$.
  If $\psi\psi^* \leq \id_Y + 2r$ then we say that $\psi$ is \emph{$r$-surjective} and if it is both then it is \emph{$r$-bijective}.
\end{dfn}

\begin{dfn}
  \label{dfn:FraisseLimit}
  Let $\cK$ be a Fraïssé class.
  By a \emph{limit} of $\cK$ we mean a separable $\cK$-structure $\fM$, satisfying that for every $\cK$-structure $\fA$, finite $A_0 \subseteq A$, $\psi \in \Stx(\fA,\fM)$ and $\varepsilon > 0$ there exists $\varphi \in \Stx^{<\psi}(\fA,\fM)$ which is $\varepsilon$-total on $A_0$.
\end{dfn}

\begin{lem}
  \label{lem:FraisseLimitCriterion}
  Let $\cK$ be a Fraïssé class, $\fM$ a separable $\cK$-structure.
  For each $n$ let $\cK_{n,0} \subseteq \cK_n$ be $d^\cK$-dense, and let $M_0 = \{a_i\}_{i \in \bN} \subseteq M$ be dense.
  We shall use the notation $a_{<m}$ for the tuple $(a_i)_{i<m}$.

  Then in order for $\fM$ to be a limit of $\cK$, is enough that for every $n,m \in \bN$, $\varepsilon > 0$, $\bar b \in \cK_{n,0}$ and $\psi\colon \bar b \times a_{<m} \rightarrow \bQ$, if $\psi \in \Stx(\bar b,\fM)$ (so in particular, $\psi\colon \bar b \rightsquigarrow a_{<m}$ is an approximate isometry) then there exist $\varphi \in \Apx^{\leq \psi}(\bar b,\fM)$ which is $\varepsilon$-total on $\bar b$.
\end{lem}
\begin{proof}
  Let $\fB$ be a $\cK$-structure, $B_0 \subseteq B$ finite, $\psi \in \Stx(\fB,\fM)$ and $\varepsilon > 0$.
  There exist a finite tuple $\bar b \in B^n$ and $\psi_0 \in \Stx(\bar b,\fM)$ such that $\psi_0 < \psi$, and we may assume that $\bar b$ contains $B_0$.
  Let $0 < \delta \leq \varepsilon/3$ be small enough that $\psi_0 + 3 \delta < \psi$.
  Let $\bar c \in \cK_{n,0}$ with $d^\cK(\bar c,\bar b) < \delta$, and let $\rho \in \Stx(\bar c,\bar b)$ witness this, namely satisfy $d(\rho) < \delta$ as per \autoref{dfn:IntrinsicDistance}.
  Then $\psi_0 \rho \in \Stx(\bar c,\fM)$, so there exists $\psi_1 \in \Apx^{<\psi_0 \rho}(\langle \bar c \rangle,\fM)$.
  Replacing $\psi_1$ with its restriction to $\bar c \times M_0$ we still have $\psi_1 < \psi_0 \rho$.
  By \autoref{lem:ApproximateIsometryStrictCoarsening}\ref{item:ApproximateIsometryStrictCoarseningCriterion} there exist some $m$ and $\psi'\colon \bar c \times a_{<m} \rightarrow \bQ$ such that $\psi_1 < \psi' < \psi_0 \rho$, i.e., $\psi' \in \Stx^{< \psi_0 \rho}(\bar c,\fM)$.
  By assumption there exists $\varphi' \in \Apx^{\leq \psi'}(\bar c,\fM)$ which is $\delta$-total on $\bar c$, and we may further assume that $\varphi' \in \Apx^{\leq \psi_0 \rho}(\bar c,a_{<k})$ for some $k$.
  Thus $\varphi'\rho^* < \psi'\rho^* + \delta \leq \psi_0\rho\rho^* + \delta \leq \psi_0 + 3\delta < \psi$, so $\psi' \rho^* + \delta \in \Stx^{<\psi}(\fB,\fM)$ and $\psi' \rho^* + \delta$ is moreover $\varepsilon$-total on $\bar b$, as desired.
\end{proof}

\begin{lem}
  \label{lem:FraisseLimitExists}
  Every Fraïssé class $\cK$ admits a limit.
\end{lem}
\begin{proof}
  We construct an increasing chain of $\fA_n \in \cK$, starting with $\fA_0$ being the unique $\emptyset$-generated structure in $\cK$, letting $i_{n,m}\colon \fA_n \rightarrow \fA_m$ denote the inclusion maps.
  For each $n$ we fix a countable $d^\cK$-dense subset of $\cK_n$, call it $\cK_{n,0}$, and a countable dense subset $A_{n,0} \subseteq A_n$, such that $A_{n,0} \subseteq A_{n+1,0}$.

  By \autoref{lem:NAP}\autoref{item:NAPStxAmalgamation} we can construct the chain $\fA_n$ so that for each $\bar b \in \cK_{n,0}$, finite subset $B \subseteq A_{m,0}$ and $\psi\colon \bar b \times B \rightarrow \bQ$, if $\psi \in \Stx(\bar b,\fA_m)$ then there exists $k > m$ and an embedding $h\colon \langle \bar b \rangle \rightarrow \fA_{k+1}$ such that $i_{k,k+1}^* h < \psi$, and in particular $h < \psi$.
  By PP and CP, the chain $\fA_0 \subseteq \fA_1 \subseteq \ldots$ admits a unique limit in the category of $\cK$-structures, which we denote by $\fM = \bigcup \fA_n$, in which $M_0 = \bigcup A_{n,0} \subseteq M$ is dense.
  By \autoref{lem:FraisseLimitCriterion}, $\fM$ is a limit.
\end{proof}

In fact, we can do better.
For $\bar a \in \cK_n$ let $[\bar a]$ denote the equivalence class $\bar a/\ker d^\cK$, and let $\overline \cK_n = \cK_n/\ker d^\cK$ denote the quotient space, equipped with the quotient metric (which is separable and complete, by PP).
For each $n$ we have a natural map $\overline \cK_{n+1} \rightarrow \overline \cK_n$, sending $[a_0,\ldots,a_n] \mapsto [a_0,\ldots,a_{n-1}]$, giving rise to an inverse system with a limit $\overline \cK_\omega = \varprojlim \overline \cK_n$, equipped with the topology induced from $\prod_n \overline \cK_n$.
A member of $\overline \cK_\omega$ will be denoted by $\xi$, represented by a compatible sequence $(\xi_n)_{n \in \bN}$.
Considering limits of increasing chains as in the proof of \autoref{lem:FraisseLimitExists}, we see that for every $\xi \in \overline \cK_\omega$ there exists a $\cK$-structure $\fM^\xi$ along with a generating sequence $\bar a^\xi = (a_i^\xi)_{i \in \bN} \subseteq M^\xi$, such that $\xi_n = [a^\xi_{<n}]$ for all $n$, and this pair $(\fM^\xi,\bar a^\xi)$ is determined by $\xi$ up to a unique isomorphism.
Conversely, any pair of a separable $\cK$-structure $\fM$ and a generating $\bN$-sequence is of this form.

\begin{thm}
  \label{thm:FraisseLimitDenseGd}
  Let $\cK$ be a Fraïssé class, and let $\overline \cK_\omega$ be as above.
  Let $\Xi$ be the set of $\xi \in \overline \cK_\omega$ for which $\fM^\xi$ is a limit of $\cK$ and every tail of the sequence $(a^\xi_i)$ is dense in $\fM^\xi$.
  Then $\overline \cK_\omega$ is a Polish space and $\Xi \subseteq \overline \cK_\omega$ is a dense $G_\delta$.
\end{thm}
\begin{proof}
  That $\overline \cK_\omega$ is a Polish space is clear.

  Let $\cK_{n,0} \subseteq \cK_n$ be countable dense as earlier, and let $\bar b \in \cK_{n,0}$, $\varepsilon > 0$ (say rational) and $\psi\colon \bar b \times m \rightarrow \bQ^{>0}$.
  Define $X_{\bar b,\varepsilon,\psi} \subseteq \overline \cK_\omega$ to consist of all $\xi$ such that one of the following holds:
  \begin{itemize}
  \item either there is no $\varphi \in \Stx(\bar b,\fM^\xi)$ such that $\varphi(b_i,a_j^\xi) < \psi(b_i,j)$ for all $i < n$, $j < m$ (let us call such a $\varphi$ \emph{good}),
  \item or there exists a good $\varphi$ such that, moreover, for each $i < n$ there is $k \geq m$ with $\varphi(b_i,a_k^\xi) < \varepsilon$.
  \end{itemize}
  It is easy to check using \autoref{lem:FraisseLimitCriterion} that $\Xi$ is the intersection of all such $X_{\bar b,\varepsilon,\psi}$, of which there are countably many, so all we need to show is that each $X_{\bar b,\varepsilon,\psi}$ is a dense $G_\delta$ set.

  The first possibility defines a closed set and the second an open one, so $X_{\bar b,\varepsilon,\psi}$ is indeed a $G_\delta$ set.
  For density, let $U \subseteq \overline \cK_\omega$ be open and $\xi \in U$.
  If there is no good $\varphi \in \Stx(\bar b,\fM^\xi)$ then $\xi \in X_{\bar b,\varepsilon,\psi} \cap U$ and we are done.
  Otherwise, let us fix a good $\varphi$, and let $\varphi_0 \in \Stx(\bar b,a_{<m}^\xi)$ be the restriction of $\varphi$ to $\bar b \times a_{<m}^\xi$.
  We may assume that $U$ is the inverse image in $\overline \cK_\omega$ of an open set $V \subseteq \overline \cK_\ell$, with $\ell \geq m$ and $\xi_\ell \in V$.
  By \autoref{lem:NAP}\autoref{item:NAPStxAmalgamation} there exists an extension $\langle a^\xi_{<\ell} \rangle \subseteq \fC \in \cK$ and an embedding $\varphi_0 > h\colon \langle \bar b \rangle \rightarrow \fC$, and we may assume that $\fC = \langle \bar c \rangle$ where $\bar c = a^\xi_{<\ell},h\bar b$, so $\bar c \in \cK_{\ell + n}$.
  Let $\zeta \in \overline \cK_\omega$ be any such that $\zeta_{\ell + n} = [\bar c]$.
  Then $\zeta \in U \cap X_{\bar b,\varepsilon,\psi}$, as desired.
\end{proof}

\begin{thm}
  \label{thm:FraisseLimitBackAndForth}
  Let $\cK$ be a Fraïssé class, $\fM$ and $\fN$ separable $\cK$-structures, and let $\psi \in \Stx(\fM,\fN)$.
  \begin{enumerate}
  \item \label{item:FraisseLimitBackAndForthEmedding}
    If $\fN$ is a limit of $\cK$ then $\psi$ strictly coarsens an embedding $\theta\colon \fM \rightarrow \fN$.
  \item \label{item:FraisseLimitBackAndForthIsomorphism}
    If both $\fM$ and $\fN$ are limits of $\cK$ then $\psi$ strictly coarsens an isomorphism $\theta\colon \fM \cong \fN$.
  \end{enumerate}
  In particular (with $\psi = \infty$), the limit of $\cK$ is unique up to isomorphism.
\end{thm}
\begin{proof}
  We only prove the second assertion, the first being similar and easier.
  Let $\{a_n\}$ and $\{b_n\}$ enumerate dense subsets of $\fM$ and $\fN$, respectively.
  We construct a decreasing sequence of $\theta_n \in \Stx(\fM,\fN)$, starting with $\theta_0 = \psi$.
  For even $n$ we choose $\theta_{n+1} \in \Stx^{<\theta_n}(\fM,\fN)$ which is $2^{-n}$-total on $a_{<n}$.
  For odd $n$ we similarly choose $\theta_{n+1} \in \Stx^{<\theta_n}(\fM,\fN)$, which is $2^{-n}$-surjective on $b_{<n}$ (i.e., $\theta_{n+1}^* \in \Stx^{<\theta_n^*}(\fN,\fM)$ which is $2^{-n}$-total on $b_{<n}$).
  Then $\theta = \lim \theta_n$ is the desired isomorphism.
\end{proof}

The unique limit of $\cK$ will be denoted by $\lim \cK$.
It can also be characterised in terms of actual maps.

\begin{cor}
  \label{cor:FraisseLimitMap}
  Let $\cK$ be a Fraïssé class and $\fM$ a separable $\cK$-structure.
  Then the following are equivalent:
  \begin{enumerate}
  \item The structure $\fM$ is a limit of $\cK$.
  \item \autoref{thm:FraisseLimitBackAndForth}\ref{item:FraisseLimitBackAndForthEmedding} holds: for every separable $\cK$-structure $\fB$ and $\psi \in \Stx(\fB,\fM)$, there is an embedding $f\colon \fB \rightarrow \fM$, $f < \psi$.
  \item For a separable $\cK$-structure $\fB$, finite tuple $\bar a \in B$, embedding $h \colon \langle \bar a \rangle \rightarrow \fM$ and $\varepsilon > 0$, there is an embedding $f\colon \fB \rightarrow \fM$ such that $d( f \bar a, h \bar a) < \varepsilon$.
    (Equivalently, we can take $h$ to be a finite partial isomorphism and $\bar a$ to enumerate $\dom h$.)
  \item \label{item:FraisseLimitMapDefinition}
    Same, where $\fB$ is finitely generated (i.e., $\fB \in \cK$).
  \end{enumerate}
\end{cor}
\begin{proof}
  \begin{cycprf}
  \item
    By \autoref{thm:FraisseLimitBackAndForth}\ref{item:FraisseLimitBackAndForthEmedding}.
  \item[\impnnext]
    Clear.
  \item[\impfirst]
    Let $\fB \in \cK$ and $\psi \in \Stx(\fB,\fM)$.
    By \autoref{lem:NAP}\autoref{item:NAPStxToEmbedding}, possibly increasing $\fB$ we may assume there is a finite partial isomorphism $h\colon \fB \dashrightarrow \fM$ such that $h < \psi$, and so $h + \varepsilon < \psi$ for some $\varepsilon > 0$.
    By hypothesis we obtain $f\colon \fB \rightarrow \fM$ such that $d(fa,ha) < \varepsilon$ for $a \in \dom h$, or $f < h + \varepsilon$.
    In particular, $f$ is total and $f < \psi$.
    Thus $\fM$ is a limit.
  \end{cycprf}
\end{proof}

\begin{thm}
  \label{thm:FraisseClass}
  Let $\cK$ be a class of finitely generated structures.
  Then the following are equivalent:
  \begin{enumerate}
  \item The class $\cK$ is a Fraïssé class.
  \item The class $\cK$ is the age of a separable approximately ultra-homogeneous structure $\fM$.
  \end{enumerate}
  Moreover, such a structure $\fM$ is necessarily a limit of $\cK$, and thus unique up to isomorphism and universal for separable $\cK$-structures.
\end{thm}
\begin{proof}
  The second item clearly implies the first, as well as the moreover part.
  Conversely, if $\cK$ is a Fraïssé class then by \autoref{lem:FraisseLimitExists} it has a limit $\fM$.
  By \autoref{thm:FraisseLimitBackAndForth}\ref{item:FraisseLimitBackAndForthEmedding} we have $\Age(\fM) = \cK$, and homogeneity follows from \autoref{thm:FraisseLimitBackAndForth}\ref{item:FraisseLimitBackAndForthIsomorphism}.
\end{proof}

\begin{rmk}
  \label{rmk:FraisseClassLanguage}
  Let $\cK$ be a Fraïssé class, and let $\theta\colon [0,\infty] \rightarrow [0,1]$ be any increasing sub-additive map which is continuous and injective near zero.
  For example, plain truncation $x \mapsto x \wedge 1$ will do, or if one wants a homeomorphism, one may take $x \mapsto 1 - e^{-x}$ or $x \mapsto \frac{x}{x + 1}$.
  The important point is that for any distance function $d$, $\theta d$ is a bounded distance function, uniformly equivalent to $d$.

  We define a new language $\cL_\cK$, consisting of one $n$-ary predicate symbol $P_{[\bar a]}$ for each equivalence class $[\bar a]$ in $\overline \cK_n$ (or in a dense subset thereof).
  Then every $\cK$-structure $\fA$ gives rise to an $\cL_\cK$-structure $\fA'$, with the same underlying set, where
  \begin{gather*}
    d^{\fA'} = \theta d^\fA, \qquad P_{[\bar a]}^\fA(\bar b) = \theta d^\cK(\bar a,\bar b).
  \end{gather*}

  Let $\cK' = \bigcup_{\fA \in \cK} \Age(\fA')$.
  Since $\cL'$ is purely relational, all members of $\cK'$ are necessarily finite, while members of $\cK$ are merely finitely generated, and in general $\cK' \neq \{\fA'\colon \fA \in \cK\}$.
  However, for each $n$ we do have canonical identification between $\cK_n$ and $\cK'_n$, with $d^{\cK'} = \theta d^\cK$.
  Then one checks that $\cK'$ is a Fraïssé class, and that a $\cK$-structure $\fM$ is a limit of $\cK$ if and only if $\fM'$ is a limit of $\cK'$.

  We conclude that up to a change of language, any Fraïssé class or approximately ultra-homogeneous structure can be assume to be in a $1$-Lipschitz, $[0,1]$-valued relational continuous language, and that our more relaxed definitions (see \autoref{rmk:Continuity}), while convenient for some concrete examples, do not in truth add any more generality.

  Another curious property of this construction is that $(\lim \cK)' = \lim \cK'$ is always an atomic model of its continuous first order theory (since all distances to types are definable), and therefore a prime model.
\end{rmk}

Notice that in \autoref{rmk:FraisseClassLanguage} all isolated types are isolated by quantifier-free formulae, but non-isolated types need not be determined by they quantifier-free restriction, so the theory need not eliminate quantifiers.

\section{Examples of metric Fraïssé classes}
\label{sec:Examples}

\subsection{Standard examples}

Let $\cK_M$ be the class of finite metric spaces; $\cK_{M,1}$ the class of finite metric spaces of diameter at most one; $\cK_H$ the class of finite dimensional Hilbert spaces; and $\cK_P$ the class of finite probability algebra, each in the appropriate language.
We leave it to the reader to check that these are all Fraïssé classes.
We claim that the Urysohn space, the Urysohn sphere, $\ell^2$, and the (probability algebra of the) Lebesgue space $([0,1],\lambda)$, are, respectively, limits of these classes.
In fact, in each of these cases, the limits satisfy a strong version of \autoref{cor:FraisseLimitMap}\ref{item:FraisseLimitMapDefinition}:
\begin{quote}
  For each extension $\fA \subseteq \fB$ of members of $\cK$, every embedding $\fA \rightarrow \fM$ extends to an embedding $\fB \rightarrow \fM$.
\end{quote}

\subsection{An incomplete example}

Fix $1 \leq p < \infty$, and let $\cK$ be the class of (real) atomic $L^p$ lattices with finitely many (see \cite{MeyerNieberg:BanachLattices} for a formal definition and \cite{BenYaacov-Berenstein-Henson:LpBanachLattices} for a model-theoretic treatment).

Then $\cK$ is \emph{not} a Fraïssé class, since it is incomplete (this is in contrast with the class of finite probability algebras, which are all atomic, and do form a complete class).
Indeed, working inside $E = L^p[0,1]$, let $f(x) = 1$ and $g(x) = x$.
Then on the one hand, $E = \langle f,g \rangle$ is non atomic, while on the other hand, approximating $g$ by step functions, the pair $(f,g)$ can be arbitrarily well approximated by pairs which do generate an atomic lattice.

The class $\cK$ is an incomplete Fraïssé class, though, and its completion is the class of \emph{all} separable $L^p$ lattices, whose limit is the unique separable atomless $L^p$ lattice.
This is somewhat uninteresting, since the limit already belongs to $\cK$.

Alternatively, one could add structure to atomic $L^p$ lattices making embeddings preserve atoms.
With this added structure, the class of $L^p$ lattices over finitely many atoms is a Fraïssé class, with limit the unique atomic $L^p$ with $\aleph_0$ atoms.
The automorphism group of the latter is $S_\infty$, the permutation group of $\bN$, so in a sense this fails to produce something truly new.

\subsection{The Gurarij space}

We recall that

\begin{dfn}
  \label{dfn:Gurarij}
  A \emph{Gurarij space} is a separable Banach space $\bG$ having the property that for any $\varepsilon > 0$, finite dimensional Banach space $E \subseteq F$, and isometric embedding $\psi\colon E \rightarrow \bG$, there is a linear embedding $\varphi\colon F \rightarrow \bG$ extending $\psi$ such that in addition, for all $x \in F$, $(1 - \varepsilon)\|x\| < \|\varphi x\| < (1 + \varepsilon)\|x\|$.
\end{dfn}

Gurarij \cite{Gurarij:UniversalPlacement} proved the existence and almost isometric uniqueness of such spaces, while actual (i.e., isometric) uniqueness of $\bG$ was shown by Lusky \cite{Lusky:UniqueGurarij}.
This uniqueness was more recently re-proved by Kubiś and Solecki \cite{Kubis-Solecki:GurariiUniqueness}, in what essentially amounts to showing that it was the Fraïssé limit of the class of all finite dimensional Banach spaces, an observation we now have the tools to state and prove formally.
From here on, $\cK = \cK_B$ is the class of finite dimensional Banach space.
Then this is a Fraïssé class.
In particular, it is separable since a separable universal Banach space exists.

Let us also recall the following fact, hitherto unpublished, due to Henson:

\begin{fct}[See also \cite{BenYaacov-Henson:Gurarij}]
  \label{fct:BanachSpaceDistance}
  Let $\bar a, \bar b \in \cK_n$.
  Then
  \begin{gather}
    \label{eq:BanachTypeDistance}
    d^\cK(\bar a,\bar b) = \sup_{\sum |s_i| = 1} \left| \bigl\| \sum s_i a_i \bigr\| - \bigl\| \sum s_i b_i \bigr\| \right|.
  \end{gather}
\end{fct}
\begin{proof}
  The inequality $\geq$ is clear.
  For $\leq$,
  let $r$ denote the right hand side of \autoref{eq:BanachTypeDistance}.
  Let $E = \langle \bar a \rangle \oplus \langle \bar b \rangle$ in the category of vector spaces over $\bR$, and for $x \in \langle \bar a \rangle$, $y \in \langle \bar b \rangle$ define:
  \begin{gather*}
    \| x - y \|' = \inf_{\bar s} \left\| x - \sum s_i a_i \right\|^{\langle \bar a \rangle} + \left\| y - \sum s_i b_i \right\|^{\langle \bar b \rangle} + r \sum |s_i|.
  \end{gather*}
  This is clearly a semi-norm on $E$, and $\|a_i - b_i\|' \leq r$.
  For $x \in \langle \bar a \rangle$ we have $\|x\|' \leq \|x\|^{\langle \bar a \rangle}$, while on the other hand, for any $\bar s$ we have by choice of $r$:
  \begin{align*}
    \|x\|^{\langle \bar a \rangle}
    &
    \leq \left\| x - \sum s_i a_i \right\|^{\langle \bar a \rangle} + \left\| \sum s_i a_i \right\|^{\langle \bar a \rangle}
    \\ &
    \leq \left\| x - \sum s_i a_i \right\|^{\langle \bar a \rangle} + \left\| \sum s_i b_i \right\|^{\langle \bar b \rangle} + r \sum |s_i|.
  \end{align*}
  It follows that $\|x\|' = \|x\|^{\langle \bar a \rangle}$, and similarly for $y \in \langle \bar b \rangle$, whence the desired amalgam.
\end{proof}

\begin{thm}
  \label{thm:GurarijExistsUnique}
  A Banach space $G$ is a Gurarij space if and only if it is the Fraïssé limit of the class of all finite dimensional Banach space.
  In particular, the Gurarij space exists, is unique, and is universal for separable Banach spaces.
\end{thm}
\begin{proof}
  Assume first that $G = \lim \cK$.
  Let $E \subseteq F$ be two finite dimensional Banach spaces, with bases $\bar a \subseteq \bar b$, respectively, and let $\psi\colon E \rightarrow G$ be an isometric embedding.
  By \autoref{cor:FraisseLimitMap} there exists an isometric $\varphi'\colon F \rightarrow G$ with $d(\bar a,\varphi\bar a) = \delta$ arbitrarily small.
  Define $\varphi\colon F \rightarrow G$ as $\psi$ on $\bar a$ and $\varphi'$ on $\bar b \setminus \bar a$.
  Taking $\delta$ sufficiently small, $\varphi$ is injective, and both $\|\varphi\|$ and $\|\varphi^{-1}\|$ (with $\varphi$ restricted to its image) arbitrarily close to one, so $G$ is Gurarij.

  Conversely, assume that $G$ is Gurarij, and let $F = \langle \bar b \rangle \in \cK$, $\psi \in \Stx(\bar b,G)$ and $\varepsilon > 0$ be given.
  By \autoref{lem:NAP}\autoref{item:NAPStxToEmbedding}, possibly extending $F$ and decreasing $\varepsilon$ we may assume that there are a finite tuple $\bar c \in F^m$ and an isometric embedding $\psi'\colon \langle \bar c \rangle \rightarrow G$ such that $\psi \geq \psi'\rest_{\bar c} + \varepsilon$.
  By assumption there exists a linear $\varphi\colon F \rightarrow G$ extending $\psi'$, with $\|\varphi\|, \|\varphi^{-1}\|$ arbitrarily close to one.
  By \autoref{fct:BanachSpaceDistance} we can then have $d^\cK\bigl( \bar b\bar c,\varphi(\bar b\bar c) \bigr) < \varepsilon$.
  Then there exists $\varphi' \in \Apx\bigl( \bar b\bar c,\varphi(\bar b\bar c) \bigr) \subseteq \Apx(F,G)$ with $\varphi'(b_i,\varphi b_i) < \varepsilon$, $\varphi'(c_j,\psi' c_j) < \varepsilon$.
  This $\varphi'$ is $\varepsilon$-total on $\bar b$ and $\psi \geq \psi'\rest_{\bar c} + \varepsilon > \varphi'\rest_{\bar c} \geq \varphi'$, so $G$ is a limit.
\end{proof}

\bibliographystyle{begnac}
\bibliography{begnac}

\end{document}